\documentclass[a4, 12pt]{amsart}
\usepackage[mathscr]{eucal}
\usepackage{amssymb}
\usepackage{latexsym}
\usepackage{amsthm}
\usepackage{color}
\theoremstyle{plain}
\newtheorem{theorem}{Theorem}[section]

\newtheorem{remark}{Remark}[section]
\newtheorem{lemma}{Lemma}[section]
\newtheorem{proposition}{Proposition}[section]

\makeatletter
\@addtoreset{equation}{section}

\title[Vacuum static space]
{$3$-dimensional complete vacuum static spaces}
\author [Q. -M. Cheng and G. Wei]{Qing-Ming Cheng and Guoxin Wei}

\address{Qing-Ming Cheng \\  \newline \indent Department of Applied Mathematics, Faculty of Science,
\newline \indent Fukuoka University, Fukuoka  814-0180, Japan.}
\email{cheng@fukuoka-u.ac.jp}

\address{Guoxin Wei \\ \newline \indent School of Mathematical Sciences, South China Normal University,
\newline \indent 510631, Guangzhou,  China.}
\email{weiguoxin@tsinghua.org.cn}

\begin{document}
\maketitle

\begin{abstract}
In this paper, we study complete Vacuum Static Spaces. A complete classification
of 3-dimensional complete Vacuum Static Spaces with non-negative scalar curvature and constant 
squared norm of Ricci curvature tensor is given by making use of the generalized maximum principle.
\end{abstract} 
\footnotetext{The first author was partially supported by JSPS Grant-in-Aid for Scientific Research:
No. 22K03303, the  fund of Fukuoka University: No. 225001.  The second author was partially supported by grant No. 12171164 of NSFC,
GDUPS (2018),  Guangdong Natural Science Foundation Grant No.2023A1515010510.}

\section{introduction}
\vskip2mm
\noindent
Let $(M^{n}, g)$ be an $n$-dimensional  Riemannian manifold with metric $g=(g_{ij})$.
If  there exists a non-constant smooth function $f$ such that  
$$
f_{,ij}=f(R_{ij}-\dfrac{R}{n-1}g_{ij})
$$
holds, $(M^n, g, f)$ is called {\it a Vacuum Static Space}, where $f_{,ij}$, $R_{ij}$ and $R$ denote components of the second covariant 
derivative of $f$, components of the Ricci curvature tensor and the scalar curvature of $(M^n, g)$.  
The function $f$ is called {\it a static potential function}.  It is known that in  many contexts, 
Riemannian manifolds admit static potential $f$. For examples, when one considers the linearization of the scalar curvature,
  static potentials $f$ are  derived.  In the general relativity, for a static space-time which carries  a perfect fluid matter field (cf. \cite{HE}, \cite{KO}),  
as  global solutions to Einstein equations, static potentials $f$ are derived.  On the other hand, existence of the static potentials $f$ is 
also very important in the problems for prescribing the scalar curvature function (cf. \cite{B}, \cite{FM}).
\noindent

It is well-known that the standard round sphere  is a Vacuum Static Space.   Furthermore, Bourguignon \cite{B} and Fischer-Marsden \cite{FM} proved
that, for a complete Vacuum Static Space $(M, g, f)$, the scalar curvature is constant. In fact, according to the second  Bianchi identity, one 
can prove that any Vacuum Static Space $(M, g, f)$ has constant scalar curvature.  On the other hand, Bourguignon \cite{B}  and Fischer-Marsden \cite{FM} 
proved
the set $f^{-1}(0)$ has the measure zero and the set $f^{-1}(0)$ is a totally geodesic regular hypersurface, that is , if $f(p)=0$, $|\nabla f(p)|^2\neq 0$. 
Since existence of a static potential
imposes many restrictions on the geometry of this underlying manifold, it might seem possible to give a description for all Vacuum Static Spaces.
In \cite{FM}, Fischer-Marsden
conjectured the following:
\vskip2mm
\noindent
{\bf Fischer-Marsden Conjecture}. The standard spheres are the only $n$-dimensional compact Vacuum Static Spaces.
\vskip2mm
\noindent
Counterexamples for the above conjecture are given by Kobayashi \cite{K} and Lafontaine \cite{L}. In fact, Kobayashi \cite{K}
gave a complete classification of $n$-dimensional complete Vacuum Static Spaces, which  are   locally conformally flat. 
\vskip2mm
\noindent
{\bf Theorem K}. {\it  A complete connected Vacuum Static Space  $(M^n, g, f)$, which is locally conformally flat, is isometric to one of the following
up to a finite quotient:
\begin{enumerate}
\item a sphere,  a Euclidean space or a hyperbolic space,
\item  the Riemannian product $\mathbb S^1\times S^{n-1}(k)$, $k>0$
\item  the Riemannian product $\mathbb R\times N(k)$, $k\neq 0$,
\item the  Warped product $\mathbb R\times_h N(1)$,
\item a warped product $\mathbb R\times_h N(k)$,
\end{enumerate}
where $N(k)$ be an $n-1$-dimensional connected complete Riemannian manifold of constant curvature $k$,
and $h$ is a periodic solution of the following equations, for constants $c_0$  and $R$, 
\begin{equation}\label{1.1}
\begin{aligned}
&h^{\prime\prime}+\dfrac{R}{n(n-1)}h=c_0h^{1-n},\\
&(h^{\prime})^2+\dfrac{2c_0}{n-2}h^{2-n}+\dfrac{R}{n(n-1)}h^2=k.
\end{aligned}
\end{equation}
}
\begin{remark} As a generalization of Kobayashi's theorems, Qing and Yuan \cite {QY} obtained  the similar results 
when  $N(k)$ is an $n-1$-dimensional  complete Einstein space with  Einstein constant $(n-2)k$, and $(M, g, f)$ is $D$-flat.
\end{remark}

\noindent
Furthermore, for $n$-dimensional, $n\geq 4$, Vacuum Static Spaces  with harmonic curvature, Kim and Shin \cite{KS} and Li \cite{Li}
have given a local classification. Besides $D$-flat Vacuum Static Spaces, there exist non-$D$-flat  Vacuum Static Spaces.
For examples, Riemannian product $S^2(\frac{R}{2(n-1)})\times N^{n-2}$ with $f=c\cos(\sqrt{\frac{R}{2(n-1)}}s)$ 
and  $H^2(\frac{R}{2(n-1)})\times N^{n-2}$ with $f=c\cosh(\sqrt{\frac{R}{2(n-1)}}s)$ are
complete Vacuum Static Spaces,  where $N^{n-2}$ is an Einstein space with Einstein constant $\frac{R}{n-1}$ and $s$ denotes
the distance function from a point on $S^2(\frac{R}{2(n-1)})$ and $H^2(\frac{R}{2(n-1)})$, respectively. Thus, we know
that there are many complete Vacuum Static Spaces.

\noindent
But for $n=3$, except the known examples of Kobayashi \cite{O}, one does not find new examples. Furthermore, Ambrozio \cite{A}
has proved the following:
\vskip2mm
\noindent
{\bf Theorem A}. {\it  A compact  Vacuum Static Space  $(M^3, g, f)$  is isometric to  one of the following up to a finite quotient if 
$$
\sum_{i,j}R_{ij}^2\leq \dfrac{R^2}{2}
$$
\begin{enumerate}
\item  the standard  sphere $S^3$,
\item the  Riemannian  product $S^1\times S^2$.
\end{enumerate}
}
\begin{remark} We should notice that the Schwarzschild-de Sitter spaces of positive mass $m$ does not satisfy the 
condition
$$
\sum_{i,j}R_{ij}^2\leq \dfrac{R^2}{2}.
$$
The Schwarzschild-de Sitter spaces of positive mass $m$  is $S^1\times_h S^2$ with $g=ds^2+h^2g_{S^2}$ and $h$ is a 
periodic solution in \eqref {1.1}.

\end{remark}
\noindent
Very recently, Xu and Ye \cite{XY} announced that they have given a complete classification for closed 3-dimensional
Vacuum  Static Spaces. Unfortunately, they used a fatal flawed inequality  in order to prove their theorem (cf. \cite{bm}).
They have retracted their article.

\vskip2mm
\noindent
In this paper, we  study 3-dimensional complete Vacuum Static Space with non-negative scalar curvature. By making use of the generalized maximum principle
due to Omori \cite{O} and Yau \cite{Y}, we prove
the following

\begin{theorem}\label{theorem 1.1}
Let $(M^{3}, g, f)$ be a $3$-dimensional complete Vacuum Static Space with non-negative  scalar curvature.
If the squared norm  $\sum_{i,j}R_{ij}^2$ of the Ricci curvature tensor is  constant, 
 $(M^3, g, f)$  is isometric to  one of the following up to a finite quotient
\begin{enumerate}
\item  the standard  sphere $S^3$,  Euclidean space $\mathbb R^3$,
\item  Riemannian  product $S^1\times S^2$,
\item Riemannian product $\mathbb R\times S^2(1)$.
\end{enumerate}
\end{theorem}

\begin{remark} It should be  remarked  that  in our theorem,  $(M^{3}, g)$
may be  complete and non-compact, which is weaker than compact in Ambrozio \cite{A} and  Xu and Ye \cite{XY}.
 $\mathbb R^3$ and $\mathbb R\times S^2(1)$ are complete non-compact  Vacuum Static Spaces.
\end{remark}


\vskip5mm
\section {Basic formulas for $n$-dimensional Vacuum Static Spaces}
\vskip5mm

\noindent
Let $(M^{n}, g)$ be an $n$-dimensional Riemannian manifold with metric $g$. For a  local orthonormal frame field
$\{e_{i}\}_{i=1}^{n}$  at a point  with dual co-frame field
$\{\omega_{i}\}_{i=1}^{n}$.
The structure equations of $M^n$ are given by
\begin{equation*}
d\omega_{i}=\sum_j \omega_{ij}\wedge\omega_j, \quad  \omega_{ij}=-\omega_{ji},
\end{equation*}
\begin{equation*}
d\omega_{ij}-\sum_k \omega_{ik}\wedge\omega_{kj}=-\frac12\sum_{k,l}R_{ijkl} \omega_{k}\wedge\omega_{l},
\end{equation*}
where $R_{ijkl}$ denote components of the curvature tensor of the  Riemannian manifold $(M^n, g)$.
Components $R_{ij}$ of Ricci curvature tensor  and scalar curvature $R$ are defined by
\begin{equation}\label{2.2}
R_{ij}=\sum_{k}R_{kikj}, \quad R=\sum_iR_{ii}.
\end{equation}
\begin{lemma} 
$R_{ijkl}$ satisfies
\begin{equation}
\begin{aligned} 
&R_{ijlk}=-R_{ijkl},  \  \   R_{jikl}=-R_{ijkl}, \\
& R_{klij}=R_{ijkl}, \\
& R_{ijkl}+R_{iklj}+R_{iljk}=0.
\end{aligned}
\end{equation}
\end{lemma}
\noindent
Defining  covariant derivative of $R_{ijkl}$ by
\begin{equation}\label{2.1-8}
\begin{aligned}
&\sum_mR_{ijkl,m}\omega_m\\
&=dR_{ijkl}+\sum_mR_{mjkl}\omega_{mi}
+\sum_mR_{imkl}\omega_{mj}+\sum_mR_{ijml}\omega_{mk}+\sum_mR_{ijkm}\omega_{ml}, 
\end{aligned}
\end{equation}
we know  that $R_{ijkl,m}$ satisfies the following
\begin{equation}\label{2.4}
\begin{aligned} 
& R_{ijkl,m}+R_{ijlm,k}+R_{ijmk,l}=0,
\end{aligned}
\end{equation}
which is called {\it the second  Bianchi identity}.
\noindent
Define  covariant derivatives  of $R_{ij}$ by
\begin{equation}\label{2.5}
\begin{aligned}
&\sum_{k}R_{ij,k}\omega_k=dR_{ij}+\sum_kR_{ik}\omega_{kj}+\sum_k R_{kj}\omega_{ki},\\
&\sum_{l}R_{ij,kl}\omega_l=dR_{ij,k}+\sum_lR_{lj,k}\omega_{li}+\sum_l R_{il,k}\omega_{lj}+\sum_l R_{ij,l}\omega_{lk},\\
\end{aligned}
\end{equation}
we have, according to the second Bianchi identity \eqref{2.4} and taking exterior derivative of the first formula of \eqref{2.5},
\begin{equation}\label{2.6}
\begin{aligned}
&R_{ik,m}-R_{im,k}=\sum_lR_{kmil,l},\\
&R_{ij,kl}-R_{ij,lk}=\sum_mR_{mj} R_{mikl}+\sum_m R_{im}R_{mjkl}.\\
\end{aligned}
\end{equation}
For a smooth function $f$, we define covariant derivatives $f_{i}, \ f_{,ij}, \ f_{,ijk}, \ f_{,ijkl}$ of $f$ by
\begin{equation*}
df =\sum_i f_{,i}\omega_i, \ \ \sum_j f_{,ij}\omega_j=df_{,i}+\sum_j
f_{,j}\omega_{ji},
\end{equation*}
\begin{equation}
\sum_j f_{,ijk}\omega_k=df_{,ij}+\sum_kf_{,kj}\omega_{ki}+\sum_lf_{,ik}\omega_{kj},
\end{equation}
\begin{equation}
\sum_j f_{,ijkl}\omega_l=df_{,ijk}+\sum_lf_{,ljk}\omega_{li}+\sum_kf_{,ilk}\omega_{lj}+\sum_kf_{,ijl}\omega_{lk},
\end{equation}
and 
\begin{equation*}
|\nabla f|^2=\sum_{i }(f_{,i})^2,\ \ \Delta f =\sum_i f_{,ii},
\end{equation*}
where $\Delta$ and $\nabla$ denote the Laplacian and the gradient
operator, respectively. We have the Ricci identities
\begin{equation}\label{2.9}
\begin{aligned}
f_{,ijk}-f_{,ikj}&=\sum_{l}f_{,l}R_{lijk},\\
f_{,ijkl}-f_{,ijlk}&=\sum_{m}f_{,mj}R_{mikl}+\sum_{m}f_{,im}R_{mjkl}.
\end{aligned}
\end{equation}
\newline
From the second  Bianchi identity \eqref{2.4}, we have
\begin{equation}\label{2.10}
\sum_iR_{ij,i}=\dfrac12R_{,j}.
\end{equation}
For $n$-dimensional Riemannian manifold $M^n$ with the metric $g$, if there exists a non-constant function $f$ such that
\begin{equation}\label{2.11}
f_{,ij}=f(R_{ij}-\dfrac{R}{n-1}g_{ij})
\end{equation}
holds, $(M^n, g, f)$ is called {\it a Vacuum Static Space}.
\begin{lemma}\label{lemma 2.2}
For    an $n$-dimensional Vacuum Static Space $(M^{n}, g, f)$, the scalar curvature is constant and
for any $i, j, k$,
\begin{equation}
\begin{aligned}
f(R_{ij,k}-R_{ik,j}) &=\sum_{l}f_{,l}R_{lijk}+(f_{,j}R_{ik}-f_{,k}R_{ij})+\dfrac{R}{n-1}(f_{,k}g_{ij}-f_{,j}g_{ik}).
\end{aligned}
\end{equation}
\end{lemma}
\begin{proof} 
Since 
\begin{equation*}
f_{,ijk}-f_{,ikj}=\sum_{l}f_{,l}R_{lijk},
\end{equation*}
according to  \eqref{2.11} and $\Delta f=-\dfrac{R}{n-1}f$,  we have
$$
f\sum_iR_{ij,i}=0.
$$
From  \eqref{2.10}, we obtain, for any $j$, 
$$
fR_{,j}=0.
$$
$R$ is constant because the set $f^{-1}(0)$ has measure  zero.
\newline
For any $i, j, k$, from \eqref{2.9} and \eqref{2.11}, we obtain
\begin{equation}
\begin{aligned}
&f(R_{ij,k}-R_{ik,j})\\
&=f_{,ijk}-f_{,ikj}-f_{,k}(R_{ij}-\dfrac{R}{n-1}g_{ij})+f_{,j}(R_{ik}-\dfrac{R}{n-1}g_{ik})\\
 &=\sum_{l}f_{,l}R_{lijk}+(f_{,j}R_{ik}-f_{,k}R_{ij})+\dfrac{R}{n-1}(f_{,k}g_{ij}-f_{,j}g_{ik}).
\end{aligned}
\end{equation}

\end{proof}
\noindent
It is known that Weyl curvature tensor  is defined by
\begin{equation}
W_{ijkl}=R_{ijkl}-\dfrac{1}{n-2}(A_{ik}\delta_{jl}+A_{jl}\delta_{ik}-A_{il}\delta_{jk}-A_{jk}\delta_{il}),
\end{equation}
where 
$$
A_{ij}=R_{ij}-\dfrac{R}{2(n-1)}\delta_{ij}.
$$
$$
C_{ijk}=R_{ij,k}-R_{ik,j}
$$ 
is called Cotton tensor. \newline
For a Vacuum Static Space $(M^n, g, f)$, 
$D_{ijk}$ is defined by
\begin{equation}
\begin{aligned}
(n-2)D_{ijk}&=(n-1)(R_{ik}f_{,j}-R_{ij}f_{,k})+R(f_{,k}\delta_{ij}-f_{,j}\delta_{ik})\\
&+\sum_l(R_{lj}\delta_{ik}-R_{lk}\delta_{ij})f_{,l}.
\end{aligned}
\end{equation}
When  $D_{ijk}\equiv 0$,  $(M^n, g, f)$ is called {\it $D$-flat}.

\noindent In order to prove our results, we need the following generalized maximum principle  due to Omori \cite{O} and Yau \cite{Y}.

\vskip2mm
\noindent
\begin{lemma}\label{lemma 2.3}
Let $(M^{n}\, g)$ be a complete Riemannian manifold  with Ricci curvature bounded from below.
For a $C^{2}$-function  $f$ bounded from above,  there exists a sequence of points $\{p_{k}\}\in M^{n}$, such that
\begin{equation*}
\lim_{k\rightarrow\infty} f(p_{k})=\sup f,\quad
\lim_{k\rightarrow\infty} |\nabla f|(p_{k})=0,\quad
\limsup_{k\rightarrow\infty}\Delta f(p_{k})\leq 0.
\end{equation*}
\end{lemma}

 \vskip10mm
\section{3-dimensional Vacuum Static Spaces}
\vskip5mm
\noindent
For $n=3$, we know that Weyl curvature tensor vanishes, that is, 
\begin{equation}\label{3.1}
R_{ijkl}=E_{ik}\delta_{jl}+E_{jl}\delta_{ik}-E_{il}\delta_{jk}-E_{jk}\delta_{il}+\dfrac{R}6(\delta_{ik}\delta_{jl}-\delta_{il}\delta_{jk}),
\end{equation}
where 
$$
E_{ij}=R_{ij}-\dfrac{R}3\delta_{ij}.
$$
$D_{ijk}$ can be rewritten
$$
D_{ijk}=2(E_{ik}f_{,j}-E_{ij}f_{,k})+\sum_l(E_{lj}\delta_{ik}-E_{lk}\delta_{ij})f_{,l}.
$$
Since 
$$
f_{,ij}=f(R_{ij}-\dfrac{R}{2}\delta_{ij})=f(E_{ij}-\dfrac{R}6\delta_{ij}),
$$
we obtain
\begin{equation}\label{3.2}
fC_{ijk}=D_{ijk}.
\end{equation}
In fact,  since the scalar curvature is constant, we have
$$
(fE_{ij})_{,k}-(fE_{ik})_{,j}=fC_{ijk}+f_{,k}E_{ij}-f_{,j}E_{ik},
$$
and 
\begin{equation}\label{3.3}
f_{,ijk}-f_{,ikj}=(fE_{ij})_{,k}-f_k\dfrac{R}6\delta_{ij}-(fE_{ik})_{,j}+f_{,j}\dfrac{R}6\delta_{ik}=\sum_lf_{,l}R_{lijk}.
\end{equation}
We obtain, from \eqref{3.1} and \eqref{3.3}, 
\begin{equation*}
\begin{aligned}
fC_{ijk}&=(fE_{ij})_{,k}-(fE_{ik})_{,j}-(f_{,k}E_{ij}-f_{,j}E_{ik})\\
&=2(E_{ik}f_{,j}-E_{ij}f_{,k})+\sum_l(E_{lj}\delta_{ik}-E_{lk}\delta_{ij})f_{,l}=D_{ijk}.
\end{aligned}
\end{equation*}
Defining 
$S=\sum_{i,j}E_{ij}^2$,  we have
\begin{equation}\label {3.4}
S=\sum_{i,j}R_{ij}^2-\dfrac{R^2}3, \ \  \sum_{i,j,k}D_{ijk}^2=8S|\nabla f|^2-12E^2(\nabla f, \nabla f),
\end{equation}
where 
$$
E^2(\nabla f, \nabla f)=\sum_{i,j,k}E_{ik}E_{kj}f_{,i}f_{,j}.
$$
Define a function $F_3$ by
$$
F_3=\sum_{i,j,k}E_{ij}E_{jk}E_{ki}.
$$
\begin{lemma}
For a 3-dimensional Vacuum Static Space, we have 
\begin{equation}
\begin{aligned}
\Delta( f_{,ij})&=6f\sum_mE_{im}E_{mj}+\dfrac{R}2fE_{ij}+\dfrac{R^2}{12}f\delta_{ij}\\
&-2fS\delta_{ij}+f_{,m}E_{mi,j}+f_{,m}C_{jmi},\\
f\Delta (E_{ij})&=6f\sum_mE_{im}E_{mj}+RfE_{ij}-2fS\delta_{ij}\\
&+f_{,m}(C_{imj}+C_{jmi})-f_{,m}E_{ijm}.\\
\end{aligned}
\end{equation}
\end{lemma}
\begin{proof}
Since 
\begin{equation}\label{3.}
\begin{aligned}
f_{,ijkl}&=f_{,ikjl}+(\sum_mf_{,m}R_{mijk})_{,l}\\
&=f_{,iklj}+\sum_{m}f_{,mk}R_{mijl}+\sum_{m}f_{,im}R_{mkjl}+\sum_mf_{,ml}R_{mijk}+\sum_mf_{,m}R_{mijk,l}\\
&=f_{,klij}+(\sum_mf_{,m}R_{mkil})_{,j}+\sum_{m}f_{,mk}R_{mijl}+\sum_{m}f_{,im}R_{mkjl}\\
&+\sum_mf_{,ml}R_{mijk}+\sum_mf_{,m}R_{mijk,l}\\
&=f_{,klij}+\sum_mf_{,mj}R_{mkil}+\sum_{m}f_{,mk}R_{mijl}+\sum_{m}f_{,im}R_{mkjl}\\
&+\sum_mf_{,ml}R_{mijk}+\sum_mf_{,m}R_{mkil,j}+\sum_mf_{,m}R_{mijk,l}.\\
\end{aligned}
\end{equation}
Since 
$$
f_{,ij}=f(E_{ij}-\dfrac{R}6\delta_{ij}), \ \ \Delta f=-\dfrac{R}2f,
$$
from \eqref{3.1}, we obtain
\begin{equation}
\begin{aligned}
\Delta (f_{,ij})&=(\Delta f)_{,ij}+\sum_mf_{,mj}R_{mi}+\sum_{m,k}f_{,mk}R_{mijk}+\sum_{m}f_{,im}R_{mj}\\
&+\sum_{m,k}f_{,mk}R_{mijk}+\sum_mf_{,m}R_{mi,j}+\sum_{m,k}f_{,m}R_{mijk,k}\\
&=6f\sum_mE_{im}E_{mj}+\dfrac{R}2fE_{ij}+\dfrac{R^2}{12}f\delta_{ij}-2fS\delta_{ij}+f_{,m}E_{mi,j}+f_{,m}C_{jmi}.\\
\end{aligned}
\end{equation}
Here we have used
\begin{equation}\label{3.}
C_{ijk}=R_{ij,k}-R_{ik,j}=\sum_{m}R_{jkim,m}.
\end{equation}
Because  $R$ is constant, we have
\begin{equation}\label{3.1-9}
\begin{aligned}
&fE_{ij,k}=f_{,ijk}+\dfrac{R}6f_k\delta_{ij}-f_{,k}E_{ij},\\
&fE_{ij,kl}=f_{,ijkl}+\dfrac{R}6f_{,kl}\delta_{ij}-f_{,kl}E_{ij}-f_{,k}E_{ij,l}-f_{,l}E_{ij,k}.
\end{aligned}
\end{equation}
Hence, we infer
\begin{equation}\label{3.1-9}
\begin{aligned}
&f\Delta (E_{ij})=\Delta(f_{,ij})-\dfrac{R^2}{12}f\delta_{ij}+\dfrac{R}2fE_{ij}-2\sum_{k}f_{,k}E_{ij,k}\\
&=6f\sum_mE_{im}E_{mj}+RfE_{ij}-2fS\delta_{ij}+\sum_mf_{,m}(C_{imj}+C_{jmi})-\sum_mf_{,m}E_{ij,m}.\\
\end{aligned}
\end{equation}
\end{proof}
\begin{proposition}\label{3.1}
For a 3-dimensional Vacuum Static Space, we have 
\begin{equation}
\begin{aligned}
\dfrac{1}2f\Delta S +\dfrac12\langle \nabla f, \nabla S\rangle &=f\bigl (\sum_{i,j,k}E_{ij,k}^2+\dfrac12\sum_{i,j,k}C_{ijk}^2+6F_3+RS\bigl),\\
\end{aligned}
\end{equation}
%
\begin{equation}
\begin{aligned}
&\dfrac{1}3f\Delta F_3 +\dfrac13\langle \nabla f, \nabla F_3\rangle \\
&=f\bigl (RF_3+S^2+2\sum_{i,j,k}E_{ik}E_{ij,m}E_{jk,m}\bigl)
+2\sum_{i,j,k,m}f_{,m}E_{ik}E_{jk}C_{imj}.\\
\end{aligned}
\end{equation}
\end{proposition}

\begin{proof} From the lemma 3.1, we have
\begin{equation}\label{3.13}
\begin{aligned}
&\dfrac12f\Delta S+\dfrac12\langle \nabla f, \nabla S\rangle=f\sum_{i,j,k}E_{ij,k}^2+f\sum_{i,j}E_{ij}\Delta (E_{ij})+\sum_{i,j,k}f_{,k}E_{ij}E_{ij,k}\\
&=f(\sum_{i,j,k}E_{ij,k}^2+6F_3+RS)-2\sum_{i,j,k}E_{ij}C_{ijk}f_{,k}.
\end{aligned}
\end{equation}
On the other hand, from \eqref{3.2} and definition of $D_{ijk}$, we have
$$
\begin{aligned}
&f\sum_{i,j,k}E_{ij}C_{ijk}f_{,k}=\sum_{i,j,k}E_{ij}D_{ijk}f_{,k}\\
&=3E^2(\nabla f, \nabla f)-2S|\nabla f|^2
\end{aligned}
$$
and 
$$
\begin{aligned}
&f^2\sum_{i,j,k}C_{ijk}^2=\sum_{i,j,k}D_{ijk}^2\\
&=-12E^2(\nabla f, \nabla f)+8S|\nabla f|^2
\end{aligned}
$$
We obtain
$$
\begin{aligned}
&4f\sum_{i,j,k}E_{ij}C_{ijk}f_{,k}=-f^2 \sum_{i,j,k}C_{ijk}^2,\\
\end{aligned}
$$
that is,
$$
\begin{aligned}
&4\sum_{i,j,k}E_{ij}C_{ijk}f_{,k}=-f\sum_{i,j,k}C_{ijk}^2,\\
\end{aligned}
$$
From \eqref{3.13}, we have
\begin{equation*}
\begin{aligned}
\dfrac{1}2f\Delta S +\dfrac12\langle \nabla f, \nabla S\rangle &=f\bigl (\sum_{i,j,k}E_{ij,k}^2+\dfrac12\sum_{i,j,k}C_{ijk}^2+6F_3+RS\bigl).\\
\end{aligned}
\end{equation*}
From the definition of $F_3$, we have
$$
\dfrac{1}{3}\nabla_mF_3=\sum_{i,j,k}E_{ik}E_{kj}E_{ij,m}
$$
$$
\dfrac{1}{3}\nabla_l\nabla_mF_3=\sum_{i,j,k}E_{ik}E_{kj}E_{ij,ml}+2\sum_{i,j,k}E_{ik}E_{kj,l}E_{ij,m}.
$$
Thus, from the lemma 3.1, we obtain
\begin{equation}\label{3.14}
\begin{aligned}
&\dfrac13f\Delta F_3+\dfrac13\langle \nabla f, \nabla F_3\rangle\\
&=f\sum_{i,j}E_{ik}E_{jk}\Delta (E_{ij})+2f\sum_{i,j,k,m}E_{ik}E_{kj,m}E_{ij,m}+\sum_{i,j,k}f_{,m}E_{ik}E_{kj}E_{ij,m}\\
&=\sum_{i,j,k}E_{ik}E_{jk}\bigl\{6f\sum_mE_{im}E_{mj}+RfE_{ij}-2fS\delta_{ij}\\
&+\sum_mf_{,m}(C_{imj}+C_{imj})-\sum_mf_{,m}E_{ij,m}\bigl\}\\
&+2f\sum_{i,j,k,m}E_{ik}E_{kj,m}E_{ij,m}+\sum_{i,j,k}f_{,m}E_{ik}E_{kj}E_{ij,m}\\
&=f(6\sum_{i,j,k, m}E_{ik}E_{kj}E_{jm}E_{mi}+RF_3-2S^2)\\
&+2f\sum_{i,j,k,m}E_{ik}E_{kj,m}E_{ij,m}-2\sum_{i,j,k}E_{ik}E_{jk}C_{ijm}f_{,m}.
\end{aligned}
\end{equation}
Since $\sum_iE_{ii}=0$, we have
$$
\sum_{i,j,k, m}E_{ik}E_{kj}E_{jm}E_{mi}=\dfrac12S^2.
$$
We obtain
\begin{equation}\label{3.14}
\begin{aligned}
&\dfrac13f\Delta F_3+\dfrac13\langle \nabla f, \nabla F_3\rangle\\
&=f(RF_3+S^2+2\sum_{i,j,k,m}E_{ik}E_{kj,m}E_{ij,m})-2\sum_{i,j,k}E_{ik}E_{jk}C_{ijm}f_{,m}.
\end{aligned}
\end{equation}

\end{proof}

\section{Proof of the theorem 1.1}
\vskip5mm
\noindent
{\it Proof of the theorem 1.1}. At each point, we choose an orthonormal frame such that 
$$
E_{ij}=\lambda_i\delta_{ij}.
$$
We know 
$$
\sum_iE_{ii}=\sum_i\lambda_i=0, \  \ S=\sum_{i,j}E_{ij}^2=\sum_i\lambda_i^2, \ \ F_3=\sum_i\lambda_i^3
$$
Hence, 
$$
-\dfrac{S^{\frac32}}{\sqrt 6}\leq F_3\leq \dfrac{S^{\frac32}}{\sqrt 6}
$$
and the equalty holds if and only if, at least, two of $\lambda_1$, $\lambda_2$ and $\lambda_3$ are equal.
Since $S$ is constant, from the proposition \ref{3.1}, we have
\begin{equation}\label{4.1}
\begin{aligned}
\sum_{i,j,k}E_{ij,k}^2+\dfrac12\sum_{i,j,k}C_{ijk}^2+6F_3+RS=0.\\
\end{aligned}
\end{equation}
Because of  $R\geq 0$,  $\sum_iE_{ii}=0$ and $S=\sum_{i,j}E_{ij}^2$, we know
$$
-\dfrac{S^{\frac32}}{\sqrt 6}\leq F_3\leq -\dfrac{RS}6.
$$
Since $S$ is constant, the Ricci curvature is bounded and $F_3$ is bounded. By applying
the generalized maximum principle due to Omori and Yau to $ F_3$, there exists a sequence
$\{p_k\} \subset  M^3$ such that 
\begin{equation}\label{4.}
\lim_{k\rightarrow\infty} F_3(p_{k})=\sup F_3,\quad
\lim_{k\rightarrow\infty} |\nabla F_3|(p_{k})=0,\quad
\limsup_{k\rightarrow\infty}\Delta F_3(p_{k})\leq 0.
\end{equation}
If 
$$
\sup F_3=-\dfrac{S^{\frac32}}{\sqrt 6},
$$
we know $F_3$ is constant. Hence, $\lambda_1$,  $\lambda_2$ and $\lambda_3$ are constant.
If $\lambda_1=\lambda_2=\lambda_3 $ holds, we know $S=0$. Hence, $C_{ijk}\equiv 0$, from Kobayashi's result \cite{K},
$M^3$ is either Euclidean space or the standard sphere.\newline
 Next, we can assume $\lambda_1=\lambda_2\neq \lambda_3 $. Since $dE_{ij}=d(\lambda_i\delta_{ij})=0$,
 from 
 $$
 \sum_kE_{ij,k}\omega_k=dE_{ij}+\sum_{k}E_{kj}\omega_{ki}+\sum_kE_{ik}\omega_{kj},
$$
we have
$$
 \sum_kE_{ij,k}\omega_k=(\lambda_j-\lambda_i)\omega_{ji}.
$$
Thus, we get, for any $i, k$, 
$$
E_{ii,k}=0, \ \ E_{12,k}=0.
$$
From $\sum_iR_{ij,i}=0$, we obtain, for any $i, j$,
$$
C_{iij}=0.
$$
Since
$$
fC_{123}=D_{123}=0.
$$
Hence, $C_{123}=0$. We have
$$
C_{ijk}=0,
$$
that is, $M^3$ is locally conformally flat. \newline
Next,  we consider the case of 
$$
-\dfrac{S^{\frac32}}{\sqrt 6}< \sup F_3\leq -\dfrac{RS}6.
$$
Thus, since $S$  is constant,  from \eqref{4.1}, $E_{ij}=\lambda_i\delta_{ij}$, $E_{ij,k}$  and $C_{ijk}$, for any $i, j, k$ are bounded.
We can assume 
\begin{equation}
\begin{aligned}
&\lim_{m\to\infty}E_{ij}(p_m)=\bar E_{ij}=\bar \lambda_i\delta_{ij}\\
&\lim_{m\to\infty}E_{ij,k}(p_m)=\bar E_{ij,k}, \  \ \lim_{m\to\infty}C_{ijk}(p_m)=\bar C_{ijk},\\
&\lim_{k\to\infty}\dfrac{\nabla f(p_k)}{|\nabla f(p_k)|}=(\bar f_{,1}, \bar f_{,2}, \bar f_{,3})
\end{aligned}
\end{equation}
Under the  limit  process, we have, for any $k$,
\begin{equation}
\begin{cases}
\begin{aligned}
& \sum_i\bar E_{ii,k}=0\\
&\sum_i\bar \lambda_iE_{ii,k}=0\\
&\sum_i\bar \lambda_i^2E_{ii,k}=0.
\end{aligned}
\end{cases}
\end{equation}
Since $\bar \lambda_1<\bar \lambda_2<\bar\lambda_3$, we get 
$$
\bar E_{ii,k}=0.
$$
Since the scalar curvature is constant, we have
$$
\sum_i\bar E_{ik,i}=0.
$$
Hence, we conclude
\begin{equation}
\begin{aligned}
\bar E_{12,1}+\bar E_{32,3}=0,  \ \bar E_{21,2}+\bar E_{31,3}=0,\  \bar E_{13,1}+\bar E_{232}=0.
\end{aligned}
\end{equation}
From
$$
fC_{ijk}=D_{ijk},
$$
and for distinct $i, j, k$ and at any point, $E_{ij}=\lambda_i\delta_{ij}$, we know, according to the definition of $D_{ijk}$,
$$
D_{ikk}=0, \ \ D_{ijk}=0.
$$ 
 Since the set $f^{-1}(0)$ has the measure zero, we can assume $f(p_k)\neq 0$ for any $k$.
Since $D_{123}(p_k)=0$, we have
$$
\bar C_{ikk}=0, \ \ \bar C_{123}=0.
$$
Hence, we conclude, for any $i, j$, 
\begin{equation}
\begin{aligned}
&\bar E_{ii,j}=0, \ \bar E_{12,3}=\bar E_{13,2}=\bar E_{23,1},\\
&\bar E_{12,1}=-\bar E_{32,3}, \  \bar E_{21,2}=-\bar E_{31,3}, \  \bar E_{13,1}=-\bar E_{23,2},
\end{aligned}
\end{equation}
According to 
$$
fC_{ijk}=D_{ijk},
$$
we know, for any $i\neq j$, 
$$
\begin{aligned}
&f(E_{iij}-E_{ij,i})(p_k)=fC_{iij}(p_k)=D_{iij}(p_k)\\
&=-(2\lambda_i+\lambda_j)f_{,j}(p_k).
\end{aligned}
$$
Hence,  we know that $\{\dfrac{f_{,j}(p_k)}{f(p_k)}\}$ is a convergent sequence from $2\bar \lambda_i+\bar \lambda_j\neq 0$. Putting 
\begin{equation}\label{4.7}
\lim_{k\to\infty}\dfrac{f_{,j}(p_k)}{f(p_k)}=a_{j},
\end{equation}
we have 
\begin{equation}\label{4.8}
\bar E_{ij,i}=(2\bar \lambda_i+\bar \lambda_j)\lim_{k\to\infty}\dfrac{f_{,j}(p_k)}{f(p_k)}=(2\bar \lambda_i+\bar \lambda_j)a_j.
\end{equation}
From 
$$
\begin{aligned}
&2E_{ik}E_{kj,l}E_{ij,l}(p_m)=\sum_{i,j,k}(\lambda_i+\lambda_j) E_{ij,k}^2(p_m)
\end{aligned}
$$
we obtain
$$
\begin{aligned}
&2\lim_{m\to\infty}\sum_{i,j,k,l}E_{ik}E_{kj,l}E_{ij,l}(p_m)\\
&=\sum_{i,j,k}(\bar\lambda_i+\bar\lambda_j)\bar E_{ij,k}^2\\
&=2(\bar\lambda_1+\bar\lambda_2)\bar E_{12,k}^2+2(\bar\lambda_1+\bar\lambda_3)\bar E_{13,k}^2+2(\bar\lambda_2+\bar\lambda_3)\bar E_{23,k}^2\\
&=-2\bar \lambda_3(\bar E_{12,1}^2+\bar E_{12,2}^2+\bar E_{12,3}^2)-2\bar \lambda_2(\bar E_{13,1}^2+\bar E_{13,2}^2+\bar E_{13,3}^2)\\
&-2\bar \lambda_1(\bar E_{23,1}^2+\bar E_{23,2}^2+\bar E_{23,3}^2)\\
&=2\bar \lambda_2\bar E_{12,1}^2+2\bar \lambda_1\bar E_{21,2}^2+2\bar \lambda_3\bar E_{13,1}^2.
\end{aligned}
$$
According to \eqref{4.8} and the above equality, we have
\begin{equation}\label{4.9}
\begin{aligned}
&2\lim_{m\to\infty}\sum_{i,j,k,l}E_{ik}E_{kj,l}E_{ij,l}(p_m)\\
&=2\bar \lambda_2\bar E_{12,1}^2+2\bar \lambda_1\bar E_{21,2}^2+2\bar \lambda_3\bar E_{13,1}^2\\
&=2\bar \lambda_2(\bar \lambda_1-\bar\lambda_3)^2a_{2}^2+2\bar \lambda_1(\bar\lambda_2-\bar \lambda_3)^2a_1^2
+2\bar \lambda_3(\bar\lambda_1-\bar\lambda_2)^2 a_3^2\\
&=\big\{2\bar \lambda_2(S-\bar \lambda_2^2)-\dfrac{4}3\sup F_3\big\}a_{2}^2
+\big\{2\bar \lambda_1(S-\bar \lambda_1^2)-\dfrac{4}3\sup F_3\big\}a_1^2\\
&+\big\{2\bar \lambda_3(S-\bar \lambda_3^2)-\dfrac{4}3\sup F_3\big\}a_3^2\\
&=2S\sum_i\bar\lambda_ia_i^2-\dfrac{4}3\sup F_3(a_1^2+a_2^2+a_3^2)-2\sum_i\bar\lambda_i^3a_i^2\\
&=S\sum_i\bar\lambda_ia_i^2-2\sup F_3(a_1^2+a_2^2+a_3^2).
\end{aligned}
\end{equation}
Here we have  used $F_3=3\lambda_1\lambda_2\lambda_3=3\lambda_i^3-\frac32\lambda_iS$ for any $i$.
\begin{equation}\label{4.10}
\begin{aligned}
&2\lim_{m\to\infty}\sum_{i,j,k,l}\dfrac{f_{,l}}{f}E_{ik}E_{jk}C_{ilj}(p_m)=2\lim_{m\to\infty}\sum_{i,j}\dfrac{f_{,j}}{f^2}\lambda_i^2D_{iji}(p_m)\\
&=2\lim_{m\to\infty}\sum_{i,j}\dfrac{f_{,j}}{f^2}\lambda_i^2\big\{2(E_{ii}f_{,j}-E_{ij}f_{,i})+\sum_l(E_{lj}\delta_{ii}-E_{li}\delta_{ij})f_{,l}\big\}(p_m)\\
&=4\sup F_3\sum_ia_i^2-6\sum_i\bar\lambda_i^3a_i^2+2S\sum_i\bar\lambda_ia_i^2\\
&=2\sup F_3\sum_ia_i^2-S\sum_i\bar\lambda_ia_i^2.
\end{aligned}
\end{equation} 
Since 
\begin{equation*}
\begin{aligned}
&\dfrac{1}3f\Delta F_3 +\dfrac13\langle \nabla f, \nabla F_3\rangle \\
&=f\bigl (RF_3+S^2+2\sum_{i,j,k}E_{ik}E_{ij,m}E_{jk,m}\bigl)
+2\sum_{i,j,k,m}f_{,m}E_{ik}E_{jk}C_{imj},\\
\end{aligned}
\end{equation*} 
from \eqref{4.}, \eqref{4.7}, \eqref{4.9} and \eqref{4.10}, we have
$$
R\sup F_3+S^2\leq 0.
$$
We have
$$
S^2\leq-R \sup F_3<-R(-\frac{S^{\frac{3}{2}}}{\sqrt{6}}),
$$
that is,
$$
S<\frac{R^2}{6}.
$$
On the other hand, from
$$
-\dfrac{S^{\frac32}}{\sqrt 6}< \sup F_3\leq -\dfrac{RS}6,
$$
we have 
$$
S>\frac{R^2}{6}
$$
It is a contradiction, which means that  this case does not occur.
We complete the proof of the theorem 1.1.

\end{document}